\documentclass[reqno]{amsart}
\usepackage{amscd,amssymb,amsmath,amsfonts,amsthm,latexsym}

\usepackage{extpfeil}
\usepackage[T1]{fontenc}
\usepackage[latin1]{inputenc}
\usepackage[all]{xy}
\usepackage{MnSymbol}
\usepackage{tikz}
\usepackage{enumitem}
\usetikzlibrary{matrix,arrows}

\def\t{\otimes}
\def \Im{\mathop{\sf Im}\nolimits}

\def \Zenter{\mathop{\sf Z}\nolimits}

\newcommand{\slie}{\rm \texttt{SLie}_{\mathbb{K}}}
\newcommand{\cross}{\rm \texttt{Cross}}

\newcommand{\cero}{\bar{0}}
\newcommand{\scero}{_{\bar{0}}}

\newcommand{\suno}{_{\bar{1}}}
\newcommand{\menosuno}[2]{(-1)^{|#1| |#2|}}

\newcommand{\acts}[1]{\mbox{}^{#1}}
\newcommand{\nh}{\mathcal{H}}

\newcommand{\xfree}{X_{M, N}}

\DeclareMathOperator{\Ker}{\sf Ker}
\DeclareMathOperator{\Coker}{\sf Coker}
\DeclareMathOperator{\Ima}{\sf Im}
\DeclareMathOperator{\Tor}{\sf Tor}

\DeclareMathOperator{\Der}{Der}
\DeclareMathOperator{\End}{End}

\DeclareMathOperator{\ab}{ab}
\DeclareMathOperator{\ad}{\sf ad}
\DeclareMathOperator{\ide}{\sf id}

\DeclareMathOperator{\class}{cl}
\DeclareMathOperator{\Ho}{H}

\DeclareMathOperator{\magm}{mag}
\DeclareMathOperator{\alg}{alg}
\DeclareMathOperator{\efe}{F}
\DeclareMathOperator{\hc}{HC}
\DeclareMathOperator{\II}{I}
\DeclareMathOperator{\VV}{V}
\DeclareMathOperator{\UU}{U}

\newtheorem{theorem}{Theorem}[section]

\newtheorem{lemma}[theorem]{Lemma}

\newtheorem{proposition}[theorem]{Proposition}
\newtheorem{cor}[theorem]{Corollary}
\newtheorem{Con}[theorem]{Construction}
\theoremstyle{definition}
\newtheorem{definition}[theorem]{Definition}
\newtheorem{example}[theorem]{Example}

\theoremstyle{remark}
\newtheorem*{remark}{Remark}

\begin{document}

\title{Non-abelian tensor product  and homology of Lie superalgebras}

\author{X.~García-Martínez, E.~Khmaladze, M.~Ladra}

\address{\small Xabier García-Martínez: \;\rm  Department of Algebra, IMAT, University of Santiago de Compostela, 15782
Santiago de Compostela, Spain}
\email{xabier.garcia@usc.es}
\address{\small Emzar Khmaladze: \;\rm A. Razmadze Mathematical
Institute of Tbilisi State University, Tamarashvili Str. 6, 0177 Tbilisi, Georgia}
\email{e.khmal@gmail.com}
\address{\small Manuel Ladra: \;\rm  Department of Algebra, IMAT, University of Santiago de Compostela, 15782
Santiago de Compostela, Spain}
\email{manuel.ladra@usc.es}

\begin{abstract}
We introduce the non-abelian tensor product of Lie superalgebras and study some of its properties.
We use it to describe the universal central extensions of Lie superalgebras.
We present the low-dimensional non-abelian homology of Lie superalgebras
and  establish its relationship with the cyclic homology of associative superalgebras.
We also define the non-abelian exterior product and give an analogue of Miller's theorem, Hopf formula and
a six-term exact sequence for the homology of Lie superalgebras.
\end{abstract}


\maketitle

\section{Introduction}
In \cite{BrLo}, Brown and Loday introduced the non-abelian tensor product of groups in the context of an application in homotopy theory.
Analogous theories of non-abelian tensor product have been developed in other algebraic structures such as Lie algebras \cite{Ell1}
and Lie--Rinehart algebras \cite{CGL}. In \cite{Ell1}, Ellis investigated the main properties of the non-abelian tensor product
of Lie algebras and its relation to the low-dimensional homology of Lie algebras. In particular, he described
the universal central extension of a perfect Lie algebra via the non-abelian tensor product.
In \cite{Ell2}, the non-abelian exterior product of Lie algebras is introduced  and  a six-term exact sequence
relating low-dimensional homologies is obtained. In \cite{Gui}, using the non-abelian tensor product,
Guin defined the non-abelian low-dimensional homology of Lie algebras and  compared these groups
with the cyclic homology and Milnor additive $K$-theory of associative algebras.

The theory of Lie superalgebras, also called $\mathbb{Z}_2$-graded Lie algebras, has aroused much interest both in mathematics and physics.
Lie superalgebras play a very important role in theoretical physics since they are used to describe supersymmetry in a mathematical framework.
A comprehensive description of the mathematical theory of Lie superalgebras is given in \cite{Kac},
containing the complete classification of all finite-dimensional simple Lie superalgebras
over an algebraically closed field of characteristic zero. In the last few years, the theory of Lie superalgebras
has experienced a remarkable evolution obtaining many results on representation theory and classification,
most of them extending well-known facts on Lie algebras.

In this paper we develop the  non-abelian tensor product and the low-dimensional non-abelian homology of Lie superalgebras,
generalizing the corresponding notions for Lie algebras, with applications in universal central extensions
and homology of Lie superalgebras and cyclic homology of associative superalgebras.

The organization of this paper is as follows: after this introduction, in Section~\ref{S:prel} we give
some definitions and necessary well-known results for the development of the paper.
We also introduce actions and crossed modules of Lie superalgebras.
In Section~\ref{S:tensor} we introduce the non-abelian tensor product of Lie superalgebras, we establish its principal properties
such as right exactness and relation with the tensor product of supermodules.
We describe the universal central extension of a perfect Lie superalgebra via the non-abelian tensor product (Theorem~\ref{T:uce}).
In particular, applying this theorem, we obtain that $\mathfrak{st}(m, n, A)$ is the universal central extension
of $\mathfrak{sl}(m, n, A)$, for $m + n \geq 5$, where $A$ is a unital associative superalgebra.
We also study nilpotency and solvability of the non-abelian tensor product of Lie superalgebras (Theorem \ref{Th_nil}).
Using the non-abelian tensor product, in Section~\ref{S:nahom} we introduce the low-dimensional non-abelian homology of Lie superalgebras
with coefficients in crossed modules. We show that, if the crossed module is a supermodule,
then the non-abelian homology is the usual homology of Lie superalgebras.
Then we apply this non-abelian homology to relate cyclic homology and Milnor cyclic homology
of associative superalgebras, extending the results of \cite{Gui}.
Finally, in the last section we construct the non-abelian exterior product of Lie superalgebras and
we use it to obtain Miller's type theorem for free Lie superalgebras, Hopf formula and
a six-term exact sequence in the homology of Lie superalgebras.

\subsection*{Conventions and notations}

Throughout this paper we denote by $\mathbb{K}$ a unital commutative ring unless otherwise stated.
All modules and algebras are defined over $\mathbb{K}$. We write $\mathbb{Z}_2=\{\bar{0},\bar{1}\}$ and use its standard field structure. We put $(-1)^{\bar{0}}=1$ and $(-1)^{\bar{1}}=-1$.

By \emph{a supermodule} $M$ we mean a module endowed with a $\mathbb{Z}_2$-gradation: $M=M_{\bar{0}}\oplus M_{\bar{1}}$.
We call elements of $M_{\bar{0}}$ (resp. $M_{\bar{1}}$) even (resp. odd). Non-zero elements of $M_{\bar{0}}\cup M_{\bar{1}}$
will be called \emph{homogeneous}. For a homogeneous $m\in M_{\bar{\alpha}}$, $\bar{\alpha}\in \mathbb{Z}_2$,
its degree will be denoted by $|m|$. We adopt the convention that whenever the degree function occurs in a formula,
the corresponding elements are supposed to be homogeneous. By \emph{a homomorphism of supermodules} $f\colon M\to N$
of degree $|f|\in \mathbb{Z}_2$ we mean a linear map satisfying $f(M_{\bar{\alpha}})\subseteq N_{\bar{\alpha}+|f|}$.
In particular, if $|f|=\bar{0}$,  then the homomorphism $f$ will be called \emph{of even grade (or even linear map)}.

By \emph{a superalgebra} $A$ we mean a supermodule $A=A_{\bar{0}}\oplus A_{\bar{1}}$ equipped with
a bilinear multiplication satisfying $A_{\bar{\alpha}}A_{\bar{\beta}}\subseteq A_{\bar{\alpha}+\bar{\beta}}$, for $\bar{\alpha}, \bar{\beta}\in \mathbb{Z}_2$.

\section{Preliminaries on Lie Superalgebras}\label{S:prel}

In this section we review some terminology on Lie superalgebras and recall notions used in the paper. We mainly
follow \cite{CCF,Mus}, although with some modifications. We also introduce notions of actions and crossed modules of Lie superalgebras.

\subsection{Definition and some examples of Lie superalgebras}

\begin{definition}
 A \emph{Lie superalgebra} is a superalgebra $M=M_{\bar{0}}\oplus M_{\bar{1}}$ with a multiplication
 denoted by $[ \ , \ ]$, called bracket operation, satisfying  the following identities:
\begin{align*}
[x, y] &= -\menosuno{x}{y}[y, x], \\
\big[x, [y, z]\big] &= \big[[x, y], z\big] + \menosuno{x}{y}\big[y, [x, z]\big],\\
[m\scero, m\scero] &= 0,
\end{align*}
for all homogeneous elements $x, y, z \in M$ and $m_{\cero} \in M\scero$.
\end{definition}

Note that the last equation is an immediate consequence of the first one in the case $2$ has an inverse in $\mathbb{K}$.
Moreover, it can be easily seen that the second equation is equivalent to the graded Jacobi identity
\[
\menosuno{x}{z} \big[x, [y, z]\big] + \menosuno{y}{x} \big[y, [z, x]\big] + \menosuno{z}{y} \big[z, [x, y]\big] = 0.
\]

For a Lie superalgebra $M=M_{\bar{0}}\oplus M_{\bar{1}}$, the even part $M_{\bar{0}}$ is a Lie algebra.
Hence, if $M_{\bar{1}}=0$, then $M$ is just a Lie algebra.  A Lie superalgebra $M$ without even part, i. e.,
$M_{\bar{0}}=0$, is an \emph{abelian Lie superalgebra}, that is, $[x, y]= 0$ for all $x, y \in M$.

A \emph{Lie superalgebra homomorphism} $f \colon M \to M'$ is a supermodule homomorphism of even grade such that $f[x, y] = [f(x), f(y)]$ for all $x, y \in M$.

\begin{example}\label{E:sup}
\begin{enumerate}[leftmargin=0.cm,itemindent=.5cm,labelsep=0.2cm,align=left]
\item []
\item[(i)] Any associative superalgebra $A$ can be considered as a Lie superalgebra with the bracket
\[
[a, b] = ab - \menosuno{a}{b} ba.
\]

\item[(ii)] Let $m$, $n$ be positive integers and $A$  a unital associative  superalgebra.  Consider
the algebra $\mathcal{M}(m,n,A)$ of all  $(m+n)\times (m+n)$-matrices with entries in $A$ and with the usual product of matrices.
A $\mathbb{Z}_2$-gradation is defined  as follows: homogeneous elements are matrices $E_{ij}(a)$ having the homogeneous element
 $a\in A$ at the position $(i,j)$ and zero elsewhere, and $|E_{ij}(a)|=|i|+|j|+|a|$, where $|i|=\bar{0}$ if $1\leq i\leq m$ and  $|i|=\bar{1}$ if $m+1\leq i\leq m+n$. With
this gradation, $\mathcal{M}(m,n,A)$ turns out to be an associative superalgebra. The corresponding Lie superalgebra will be denoted by $\mathfrak{gl}(m,n,A)$.

\item[(iii)] Let $V=V\scero \oplus V\suno$ be a supermodule. Then the supermodule $\End_{\mathbb{K}}(V)$ of all linear endomorphisms $V\to V$ (of both degrees $0$ and $1$)  has a structure of an associative superalgebra with respect to composition (see \cite{CCF}) and hence becomes a Lie superalgebra. In particular, if  the ground ring $\mathbb{K}$ is a field, and $m$, $n$ are dimensions of $V\scero$ and $V\suno$ respectively, then choosing a homogeneous basis of $V$ ordered such that even elements stand before odd, the elements of $\End_{\mathbb{K}}(V)$ can be seen as $(m+n)\times (m+n)$-square matrices
\[
\begin{pmatrix}
a & b \\
c & d
\end{pmatrix}
\]
where $a, b, c$ and $d$ are respectively $m\times m$, $m \times n$, $n \times m$ and $n \times n$ matrices with entries in $\mathbb{K}$.
The even elements are the matrices with $b = c = 0$ and the odd elements are matrices with $a = d = 0$.
\end{enumerate}
\end{example}

Let $M$ and $N$ be two submodules of a Lie superalgebra $P$. We denote by $[M, N]$ the submodule of $P$ spanned
by all elements $[m, n]$ with $m \in M$ and $n \in N$. A $\mathbb{Z}_2$-graded submodule $M$ is a \emph{graded ideal}
of $P$ if $[M, P] \subseteq M$. In particular, the submodule $\Zenter (P) =  \{ c \in P : [c, p] = 0 \text{ for all } p \in P \}$
is a graded ideal and it is called the \emph{centre} of $P$. Clearly if $M$ and $N$ are graded ideals of $P$, then so is $[M, N]$.

Let $M$ be a Lie superalgebra and $D \in \End_{\mathbb{K}}(M)$. We say that $D$ is a \emph{derivation} if for all $x, y \in M$
\[
D([x, y]) = [D(x), y] + \menosuno{D}{x}[x, D(y)].
\]
We denote by $\big(\Der_{\mathbb{K}}(M)\big)_{\bar{\alpha}}$ the set of homogeneous derivations of degree $\bar{\alpha}\in \mathbb{Z}_2$.
 One verifies that the supermodule of derivations
\[
\Der_{\mathbb{K}}(M) = \big(\Der_{\mathbb{K}}(M)\big)\scero \oplus \big(\Der_{\mathbb{K}}(M)\big)\suno
\]
 is a subalgebra of the Lie superalgebra $\End_{\mathbb{K}}(M)$.

\subsection{Actions and crossed modules of Lie superalgebras}
\begin{definition}
Let $P$ and $M$ be two Lie superalgebras. By an \emph{action} of $P$ on $M$ we mean a $\mathbb{K}$-bilinear map of even grade,
\[
P \times M \to M, \quad (p,m) \mapsto \ ^pm,
\]
such that
\begin{enumerate}
\item[(i)]   $ \quad \! ^{[p,p']}m = \ ^{p}({}^{p'}\!m) - \menosuno{p}{p'} \ ^{p'}(^{p}m), $
\item[(ii)]   $ ^{p}[m,m'] = [ ^{p}m, m'] + \menosuno{p}{m} [m, {}^{p}m' ], $
\end{enumerate}
for all homogeneous $p, p' \in P$ and $m, m' \in M$.

The action is called \emph{trivial} if $^{p}m=0$ for all $p\in P$ and $m\in M$.
\end{definition}

For example, if $M$ is a graded ideal and $P$ is a subalgebra of a Lie superalgebra $Q$, then the bracket in $Q$ induces an action of $P$ on $M$.

Note that the action of $P$ on $M$ is the same as a Lie superalgebra homomorphism $P \to \Der_{\mathbb{K}}(M)$.

\begin{remark} If $M$ is an abelian Lie superalgebra enriched with an action of a Lie superalgebra $P$,
then $M$ has a structure of a supermodule over $P$ ($P$-supermodule, for short) (see e. g. \cite{Mus}),
that is, there is a $\mathbb{K}$-bilinear map of even grade $P \times M \to M$, $(p, m) \mapsto pm$, such that
\[
[p, p']m = p(p'm) - \menosuno{p}{p'} p'(pm),
\]
for all homogeneous $p, p' \in P$ and $m \in M$.
\end{remark}

Note that a $P$-supermodule $M$ is the same as a $\mathbb{K}$-supermodule $M$ together with a Lie superalgebra homomorphism $P \to \End_{\mathbb{K}}(M)$.

\begin{definition}
Given two Lie superalgebras $M$ and $P$ with an action of $P$ on $M$, we can define the \emph{semidirect product $M \rtimes P$}
with the underlying supermodule $M \oplus P$ endowed with the bracket given by
\[
[(m, p), (m', p')] = ([m, m'] + \acts{p}m' - \menosuno{m}{p'} ({}\acts{p'}\!m), [p, p']).
\]
\end{definition}

Now we are ready to introduce the following notion of crossed modules of Lie superalgebras (see also \cite[Definition 5]{ZL}).
\begin{definition}\label{D:crossed}
A \emph{crossed module of Lie superalgebras} is a homomorphism of Lie superalgebras $\partial \colon M \to P$ with an action of $P$ on $M$ satisfying
\begin{enumerate}
\item[(i)] $ \; \partial(^p m) = [p, \partial(m)]$,
\item[(ii)] $\acts{\partial(m)} m' = [m, m']$,
\end{enumerate}
for all $p \in P$ and $m, m' \in M$.
\end{definition}

\begin{example}
There are some standard examples of crossed modules:
\begin{enumerate}[leftmargin=0.cm,itemindent=.5cm,labelsep=0.2cm,align=left]
\item[(i)] The inclusion $ M \hookrightarrow P$ of a graded ideal $M$ of a Lie superalgebra $P$ is a crossed module of Lie superalgebras.
\item[(ii)] If $P$ is a Lie superalgebra and $M$ is a $P$-supermodule, the trivial map $0 \colon M \to P$ is a crossed module of Lie superalgebras.
\item[(iii)] A central extension  of Lie superalgebras $\partial \colon M \twoheadrightarrow P$ (i.e., $\Ker \partial \subseteq \Zenter(M)$)
 is a crossed module of Lie superalgebras. Here the action of $P$ on $M$ is given by $\acts{p}m = [\bar{m}, m]$, where $\bar{m} \in M$ is any element of $\partial^{-1}(p)$.
\item[(iv)] The homomorphism of Lie superalgebras $\partial \colon M \to \Der_{\mathbb{K}}(M)$ which sends $m \in M$
to the inner derivation $\ad (m) \in \Der_{\mathbb{K}}(M)$, defined by $\ad (m)(m') = [m, m']$, together with the action of
 $\Der_{\mathbb{K}}(M)$ on $M$ given by ${}^D m = D(m)$, is a crossed module of Lie superalgebras.
\end{enumerate}
\end{example}

\begin{lemma}\label{L:crosker}
Let $\partial \colon M \to P$ be a crossed module of Lie superalgebras. Then the following conditions are satisfied:
\begin{itemize}[font=\normalfont]
\item[(i)] The kernel of $\partial$ is in the centre of $M$.
\item[(ii)] The image of $\partial$ is a graded ideal of $P$.
\item[(iii)] The Lie superalgebra $\Ima \partial$ acts trivially on the centre $\Zenter(M)$, and so trivially
 on $\Ker \partial$. Hence $\Ker \partial$ inherits an action of $P/\Ima \partial$ making $\Ker \partial$ a $P / \Ima \partial$-supermodule.
\end{itemize}
\end{lemma}

\begin{proof}
This is an immediate consequence of Definition~\ref{D:crossed}.
\end{proof}

\subsection{Free Lie superalgebra and enveloping superalgebra of a Lie superalgebra}

\begin{definition}
The \emph{free Lie superalgebra} on a $\mathbb{Z}_2$-graded set $X = X\scero \cup X\suno$ is a Lie superalgebra $\efe(X)$ together
 with a degree zero map $i \colon X \to \efe(X)$ such that if $M$ is any Lie superalgebra and $j \colon X \to M$ is a degree zero map,
  then there is a unique Lie superalgebra homomorphism $h \colon \efe(X) \to M$ with $j = h \circ i$.
\end{definition}

The existence of free Lie superalgebras is guaranteed by an analogue of Witt's theorem (see \cite[Theorem 6.2.1]{Mus}).
In the sequel we need the following construction of the free Lie superalgebra.

\begin{Con}\label{C:free}
Let $X = X\scero \cup X\suno$ be a $\mathbb{Z}_2$-graded set. Denote by  $\magm(X)$ the free magma over the set $X$.
The free superalgebra  on $X$, denoted by $\alg(X)$, has as elements the finite sums
$\sum_i \lambda_i v_i$, where $\lambda_i \in \mathbb{K}$ and $x_i$ are elements of $\magm(X)$
and the multiplication in $\alg(X)$ extends the multiplication in $\magm(X)$. Note that the grading is naturally defined in $\alg(X)$.
The free Lie superalgebra $\efe(X)$ is the quotient $\alg(X)/I$, where $I$ is the graded ideal generated by the elements
\begin{gather*}
xy + \menosuno{x}{y}yx, \\
\menosuno{x}{z}\big(x(yz)\big) + \menosuno{y}{x}\big(y(zx)\big) + \menosuno{z}{y}\big(z(xy)\big), \\
x\scero x\scero,
\end{gather*}
for all homogeneous $x, y, z \in X$ and $x_{\cero} \in X\scero$.
\end{Con}

\begin{definition}
 The \emph{universal enveloping superalgebra} of a Lie superalgebra $M$ is a pair $(\UU(M), \sigma)$, where $\UU(M)$ is
 a unital associative superalgebra and $\sigma \colon M \to \UU(M)$ is an even  linear map satisfying
\begin{equation}\label{E:enveloping}
\sigma[x, y] = \sigma(x)\sigma(y) - \menosuno{x}{y}\sigma(y)\sigma(x),
\end{equation}
for all homogeneous $x, y \in M$, such that the following universal property holds: for any other pair $(A, \sigma')$,
where $A$ is a unital associative superalgebra and $\sigma' \colon M \to A$ is an even linear map satisfying \eqref{E:enveloping},
there is a unique superalgebra homomorphism $f \colon \UU(M) \to A$ such that $f \circ \sigma = \sigma'$.
\end{definition}

Now we need to recall (see e. g. \cite{Sch}) that, given two supermodules $M$ and $N$,
the tensor product of modules $M \t_{\mathbb{K}} N$ has a natural supermodule structure with   $\mathbb{Z}_2$-grading given by
\[
(M \t_{\mathbb{K}} N)_{\bar{\alpha}} = \bigoplus_{\bar{\beta} + \bar{\gamma} = \bar{\alpha}} (M_{\bar{\beta}} \t_{\mathbb{K}} N_{\bar{\gamma}}).
\]
In particular, the tensor power $M^{\otimes n}$, $n\geq 2$, has the induced $\mathbb{Z}_2$-grading. Hence the tensor algebra $T(M)$ has the $\mathbb{Z}_2$-grading extending that of $M$. We call $T(M)$ \emph{the tensor superalgebra}.

\begin{Con}\label{C:enveloping}
Let $M$ be a Lie superalgebra and  $T(M)$  the tensor superalgebra over the underlying supermodule of $M$.
 Consider the two-sided ideal $J(M)$ of $T(M)$ generated by all elements of the form
\[
m\t m' - \menosuno{m}{m'}m'\t m -[m,m'],
\]
for all homogeneous $m,m'\in M$. Then the quotient $\UU(M)=T(M)/J(M)$ is a unital associative  superalgebra.
By composing the canonical inclusion $M\to T(M)$ with the canonical projection $T(M)\to\UU(M)$ we get
the  canonical even linear map $\sigma\colon M\to\UU(M)$. Then the pair $(\UU(M), \sigma)$ is the universal enveloping superalgebra of $M$ (see \cite{CCF}).
\end{Con}

Note that, as in the Lie algebra case, the universal enveloping superalgebra turns out to be
a very useful tool for the representation theory of Lie superalgebras. In particular, by the universal property,
 it follows that a Lie supermodule over a Lie superalgebra $M$ is the same as
  a $\mathbb{Z}_2$-graded (left) $\UU(M)$-module (see \cite[Chapter 1]{Sch}).

Let us consider $\mathbb{K}$ with $\mathbb{Z}_2$-grading concentrated in degree zero, that is, with $\mathbb{K}_{\bar{1}}=0$.
Then the trivial map from a Lie superalgebra $M$ into $\mathbb{K}$ gives rise to a unique homomorphism
of superalgebras $\varepsilon \colon \UU(M)\to \mathbb{K}$.
The kernel of $\varepsilon$, denoted by $\Omega(M)$, is called the \emph{augmentation ideal} of $M$.
Obviously, $\Omega(M)$ is just the graded ideal of $\UU(M)$ generated by $\sigma(M)$.

\subsection{Homology of Lie superalgebras}
Now we briefly recall from \cite{Mus, Tan} the definition of homology of Lie superalgebras.

The \emph{Grassmann algebra} of a Lie superalgebra $P$,  denoted by $\bigwedge_{\mathbb{K}}(P)$,
is defined to be the quotient of the tensor superalgebra $T(P)$ of $P$ by the ideal generated by the elements
\[
x \t y + \menosuno{x}{y} y \t x,
\]
for all homogeneous $x, y \in P$. Note that $ \bigwedge_{\mathbb{K}}(P) = \bigoplus_{n>0}\bigwedge_{\mathbb{K}}^n(P)$, where $\bigwedge_{\mathbb{K}}^n(P)$ is the image of $P^{\otimes n}$ in $\bigwedge_{\mathbb{K}}(P)$,  has an induced $P$-supermodule structure given by
\[
x(x_1 \wedge \cdots \wedge x_n) = \sum_{i=1}^n (-1)^{|x|\sum_{k<i}|x_k|}(x_1 \wedge \cdots \wedge [x, x_i] \wedge \cdots \wedge x_n).
\]

Let $M$ be a $P$-supermodule and consider the chain complex $(C_{*}(P, M), d_{*})$ defined by
$C_n(P, M) = \bigwedge_{\mathbb{K}}^n(P) \t_{\mathbb{K}} M$, for $n \geq 0$, with boundary maps
$d_n \colon C_n(P, M) \to C_{n-1}(P, M)$ defined on generators by
\begin{align*}
d_n&(x_1 \wedge \cdots \wedge x_n \t y)=
\sum_{i=1}^n (-1)^{i + |x_i|\sum_{k>i}|x_k|}(x_1 \wedge \cdots \wedge \hat{x}_i \wedge \cdots \wedge x_n \t x_iy) \\
{} &+ \sum_{i<j} (-1)^{i + j + |x_i|\sum_{k<i}|x_k| + |x_j|\sum_{l<j}|x_l| + |x_i||x_j|}([x_i, x_j] \wedge
\cdots \wedge \hat{x}_i \wedge \cdots \wedge \hat{x}_j \wedge \cdots \wedge x_n \t y).
\end{align*}
The $n$-th \emph{homology of the Lie superalgebra} $P$ with coefficients in the $P$-supermodule $M$, $\Ho_n(P, M)$, is the $n$-th homology of the chain complex $(C_{*}(P, M), d_{*})$, i.e.
\[
\Ho_n(P, M) = \dfrac{\Ker d_{n}}{\Ima d_{n+1}}.
\]
If $\mathbb{K}$ is regarded as a trivial $P$-supermodule, we write $\Ho_n(P)$ for $\Ho_n(P, \mathbb{K})$.

In the case when the ground ring $\mathbb{K}$ is a field, there is a relation between $\Tor$ functor and the homology (see \cite{Mus}) given by
\[
\Ho_n(P, M) \cong \Tor_n^{\UU(P)}(\mathbb{K}, M).
\]
By analogy to Lie algebras (see e. g. \cite{HiSt}), we have the following isomorphisms
\begin{equation}\label{E:zero_homology}
H_0(P,M)\cong \Coker\big(\Omega(P)\t_{\UU(P)}M\longrightarrow M\big),
\end{equation}
\begin{equation}\label{E:first_homology}
H_1(P,M)\cong \Ker\big(\Omega(P)\t_{\UU(P)}M\longrightarrow M\big).
\end{equation}

\section{Non-abelian tensor product of Lie superalgebras}\label{S:tensor}
In this section we introduce a non-abelian tensor product of Lie superalgebras, which generalizes the non-abelian tensor product of Lie algebras \cite{Ell1},
and study its properties.

\subsection{Construction of the non-abelian tensor product}

\begin{definition}\label{D:tensor}
Let $M$ and $N$ be two Lie superalgebras with actions on each other. Let $\xfree$ be the $\mathbb{Z}_2$-graded set of all symbols
$m \t n$, where $m \in M\scero \cup M\suno$, $n \in N\scero \cup N\suno$ and the $\mathbb{Z}_2$-gradation is given by
$|m\t n| = |m| + |n|$. We define the \emph{non-abelian tensor product} of $M$ and $N$, denoted by $M \t N$,
as the Lie superalgebra generated by $\xfree$ and subject to the relations:
\begin{enumerate}
\item[(i)] $\lambda(m \t n) = \lambda m \t n = m \t \lambda n$,
\item[(ii)] $\begin{aligned}[t]
    (m + m') \t n &= m \t n + m' \t n, \quad \ \text{where}  \ m, m' \ \text{have the same grade}, \\
    m \t (n + n') &= m \t n + m \t n', \quad \ \text{where}  \ n, n' \ \text{have the same grade},
  \end{aligned}$
\item[(iii)] $\begin{aligned}[t]
    [m, m'] \t n & = m \t {}\acts{m'}\!n - \menosuno{m}{m'} (m' \t \acts{m} n), \\
          m \t [n, n'] &= (-1)^{|n'|(|m|+|n|)} ({}\acts{n'}\!m \t n) - \menosuno{m}{n} (\acts{n}m \t n'),
  \end{aligned}$
\item[(iv)] $[m \t n, m' \t n'] = - \menosuno{m}{n} (^n m \t {}\acts{m'}\!n')$,
\end{enumerate}
for every $\lambda \in \mathbb{K}$, $m, m' \in M\scero \cup M\suno$ and $n, n' \in N\scero \cup N\suno$.
\end{definition}

Let us remark that if $m = m\scero + m\suno$ is any element of $M$ and $n = n\scero + n\suno$ is any element of $N$, then under the notation $m \t n$ we mean the sum
\[
m\scero \t n\scero + m\scero \t n\suno + m\suno \t n\scero + m\suno \t n\suno.
\]

If $M = M\scero$ and $N = N\scero$ then $M \t N$ is the non-abelian tensor product of Lie algebras introduced and studied in \cite{Ell1} (see also \cite{InKhLa}).

\begin{definition}
Actions of Lie superalgebras $M$ and $N$ on each other are said to be \emph{compatible} if
\begin{enumerate}
\item[(i)] $ \; ^{(^{n}m)}n' =-\menosuno{m}{n}[\acts{m}n, n']$,
\item[(ii)] $^{(^{m}n)}m' =-\menosuno{m}{n}[\acts{n}m, m']$,
\end{enumerate}
for all $m, m' \in M\scero \cup M\suno$ and $n, n' \in N\scero \cup N\suno$.
\end{definition}
For example, if $M$ and $N$ are two graded ideals of some Lie superalgebra, the actions induced by the bracket are compatible.

\begin{proposition}\label{P:tensoriso}
Let $M$ and $N$ be Lie superalgebras acting compatibly on each other. Then there is a natural isomorphism of Lie superalgebras
\[
M \t N \cong \dfrac{M \t_{{\mathbb{K}}} N}{D(M, N)},
\]
where $D(M, N)$ is the submodule of the supermodule $M \t_{\mathbb{K}} N$ generated by the elements
\begin{enumerate}[font=\normalfont]
\item[(i)] $[m, m'] \t n - m \t \acts{m'}\!n + \menosuno{m}{m'} (m' \t \acts{m} n)$,
\item[(ii)] $m \t [n, n'] - (-1)^{|n'|(|m|+|n|)} (\acts{n'}\!m \t n) + \menosuno{m}{n} (\acts{n}m \t n')$,
\item[(iii)] $(\acts{n}m) \t (\acts{m}n)$, with $|m| = |n|$
\item[(iv)] $(-1)^{|m||n|}(\acts{n}m) \t (\acts{m'}\!n') + (-1)^{(|m|+|n|)(|m'|+|n'|) + |m'||n'|} (\acts{n'}\!m') \t (\acts{m}n)$,
\item[(v)] $\underset{(m, n), (m', n'), (m'', n'')}{\text{\Huge$\circlearrowleft$}}(-1)^{(|m|+|n|)(|m''|+|n''|) + |m||n| + |m'||n'|} [^n m, \acts{n'}\!m'] \t (\acts{m''}\!n'') $,
\end{enumerate}
for all $m, m', m'' \in M\scero \cup M\suno$ and $n, n', n'' \in N\scero \cup N\suno$,
 where $\underset{x, y, z}{\text{\Huge$\circlearrowleft$}}$ denotes the cyclic summation with respect to $x,y,z$.
\end{proposition}

\begin{proof}
There is a Lie superalgebra structure on the supermodule $(M \t_{\mathbb{K}} N)/D(M, N)$ given on generators by the following bracket
\[
[m \t n, m' \t n'] = - \menosuno{m}{n} (^n m \t \acts{m'}\!n'),
\]
for all $m, m' \in  M\scero \cup M\suno$, $n, n' \in N\scero \cup N\suno$ and extended by linearity.
 It is routine to check that this bracket is compatible with the defining relations of $(M \t_{\mathbb{K}} N)/D(M, N)$
  and it indeed  defines a Lie superalgebra structure. Then the canonical homomorphism $M \t N \to (M \t_{\mathbb{K}} N)/D(M, N)$, $m \t n \mapsto m \t n$, is an isomorphism.
\end{proof}

The proof of the following proposition is a routine calculation.
\begin{proposition}\label{P:epim}
Let $M$ and $N$ be two Lie superalgebras acting compatibly on each other.
\begin{itemize}[font=\normalfont]
\item[(i)] The following morphisms
\begin{align*}
\mu &\colon M \t N \to M,  \quad  m \t n \mapsto - \menosuno{m}{n} (\acts{n}m), \\
\nu &\colon M \t N \to N,   \quad \;  m \t n \mapsto \acts{m}n,
\end{align*}
are Lie superalgebra homomorphisms.

\item[(ii)] There are actions of $M$ and $N$ on $M \t N$ given by
\begin{align*}
\acts{m'}(m \t n) &= [m', m] \t n + \menosuno{m}{m'} m \t (^{m'}\!n), \\
\acts{n'}(m \t n) &= (^{n'}\!m) \t n + \menosuno{n}{n'} m \t [n', n],
\end{align*}
for $m, m' \in M\scero \cup M\suno$, $n, n' \in N\scero \cup N\suno$ and extended by linearity.
Moreover, with these actions $\mu$ and $\nu$ are crossed modules of Lie superalgebras.
\end{itemize}
\end{proposition}

We will denote by $[M, N]^M$ (resp. $[M, N]^N$) the image of $\mu$ (resp. $\nu$), which by Lemma~\ref{L:crosker}(ii)
 is a graded ideal of $M$ (resp. $N$) generated by the elements of the form ${}^n m$ (resp. ${}^m n$) for $m \in M$
 and $n \in N$. Note that by Lemma~\ref{L:crosker}(iii) $\Ker(\mu)$ \big(resp. $\Ker(\nu)$\big) is an $M / [M, N]^M$-supermodule (resp. $N/[M, N]^N$-supermodule).

\subsection{Some properties of the non-abelian tensor product}

The obvious analogues of Brown and Loday results~\cite{BrLo} hold for Lie superalgebras.
 In the following two propositions immediately below we show that sometimes the non-abelian tensor product of Lie superalgebras
can be expressed in terms of the  tensor product of supermodules.
\begin{proposition}\label{P:ab}
Let $M$ and $N$ be Lie superalgebras acting on each other. Then the canonical map
 $M \otimes_{\mathbb{K}}  N \to M \t N$, $m \t n \mapsto m \t n$, is an even, surjective homomorphism of supermodules.
  In addition, if $M$ and $N$ act trivially on each other, then $M \t N$ is an abelian Lie superalgebra and there is an isomorphism of supermodules
\[
M \t N \cong M^{\ab} \otimes_{\mathbb{K}} N^{\ab},
\]
where $M^{\ab} = M/[M, M]$ and $N^{\ab} = N / [N, N]$.
\end{proposition}

\begin{proof}
It is straightforward by the identities (iv), (iii) of Definition~\ref{D:tensor}.
\end{proof}

\begin{proposition}\label{P:augmentation} Let $P$ be a Lie superalgebra and $M$  a $P$-supermodule considered as an abelian Lie superalgebra
acting trivially on $P$. Then there is an isomorphism of supermodules
\[
 P\t M \cong \Omega(P) \t_{\UU(P)} M.
\]
\end{proposition}

\begin{proof}
By Proposition~\ref{P:tensoriso} there is an isomorphism of supermodules
\[
P\t M \cong \frac{P\t_{\mathbb{K}} M}{W},
\]
where $W$ is the submodule of $P\t_{\mathbb{K}} M$ generated by all elements of the form
\[
[p,p']\t m -p\t p'm +\menosuno{p}{p'} p'\t pm
\]
for all $p,p'\in P\scero \cup P\suno$ and $m\in M\scero \cup M\suno$. Now by using Construction~\ref{C:enveloping}
 and by repeating the respective part of the proof of \cite[Proposition 13]{CaLa}, it is easy to see that there is an isomorphism of supermodules
\[
\frac{P\t_{\mathbb{K}} M}{W}\cong \Omega(P) \t_{\UU(P)} M,
\]
which completes the proof.
\end{proof}

The non-abelian tensor product of Lie superalgebras is symmetric, in the sense of the following proposition.
\begin{proposition}
The Lie superalgebra homomorphism
\[
M \t N \to N \t M, \quad m \t n \mapsto - \menosuno{m}{n} (n \t m),
\]
is an isomorphism.
\end{proposition}

\begin{proof}
This can be checked readily.
\end{proof}

Let us consider the category $\slie^2$ whose objects are ordered pairs of Lie superalgebras $(M, N)$ acting compatibly on each other,
and the morphisms are pairs of Lie superalgebra homomorphisms $(\phi \colon M \to M',\psi \colon N \to N')$ which preserve the actions,
i.e., $\phi(^n m) = \acts{\psi(n)}\phi(m)$ and $\psi(^m n) = \acts{\phi(m)}\psi(n)$. For such a pair $(\phi, \psi)$
we have a homomorphism of Lie superalgebras $\phi \t \psi \colon M \t N \to M' \t N'$,  $m \t n \mapsto \phi(m) \t \psi(n)$.
Therefore, $\otimes$ is a functor from $\slie^2$ to the category of Lie superalgebras.

Given an exact sequence in $\slie^2$
\begin{equation}\label{E:seq}
(0, 0) \xrightarrow{ \  \; \ } (K, L) \xrightarrow{(i, j)}  (M, N) \xrightarrow{(\phi, \psi)} (P, Q) \xrightarrow{ \  \; \ } (0, 0),
\end{equation}
by Proposition~\ref{P:epim}(ii) there is a Lie superalgebra homomorphism $M \t L \to L$ and an action of $N$ on $K \t N$.
Thus, there is an action of $M \t L$ on $K \t N$, so we can form the semidirect product $(K \t N) \rtimes (M \t L)$,
 and we have the following obvious analogue of \cite[Proposition 9]{Ell1}.

\begin{proposition}\label{P:tenfunc}
Given  the short exact sequence \eqref{E:seq}, there is an exact sequence of Lie superalgebras
\[
(K \t N) \rtimes (M \t L) \xrightarrow{ \ \alpha \ }  M \t N \xrightarrow{\phi \t \psi}  P \t Q \xrightarrow{ \  \; \ } 0.
\]
\end{proposition}

In particular, given a Lie superalgebra $M$ and a graded ideal $K$ of $M$,  there is an exact sequences of Lie superalgebras
\begin{equation}\label{C:exact}
(K \t M) \rtimes (M \t K) \to M \t M \to (M / K) \t (M / K) \to 0.
\end{equation}

\subsection{Nilpotency, solvability and Engel of the non-abelian tensor product}

The results from \cite{STME} on nilpotency, solvability and Engel  of the non-abelian tensor product on Lie algebras can be easily extended to the case of Lie superalgebras.
 The notions of nilpotency and solvability  of Lie superalgebras are given in \cite{Mus}.  As they are very similar to the respective notions for Lie algebras, we omit them.
We say that a Lie superalgebra $M$  is  \emph{$n$-Engel }if it satisfies  $\ad(x)^n=0$ for all $x \in M$.
The proof of the following result is similar to the proof of \cite[Theorem 2.2]{STME}.

\begin{theorem}\label{Th_nil}
Let $M$ and $N$ be two Lie superalgebras acting compatibly on each other. Then,
\begin{enumerate}[font=\normalfont]
\item[(i)] If $[M, N]^M$ is nilpotent, then $M \t N$ and $[M, N]^N$ are nilpotent too. Moreover, if the nilpotency class of $[M, N]^M$ is $\class([M, N]^M)$, then
\begin{align*}
\class([M, N]^M) \leq \class(M \t N) \leq \class([M, N]^M) + 1, \\
\class([M, N]^N) \leq \class([M, N]^M) + 1.
\end{align*}

\item[(ii)] If $[M, N]^M$ is solvable, then $M \t N$ and $[M, N]^N$ are solvable too. Moreover, if the derived length of $[M, N]^M$ is $\ell([M, N]^M)$, then
\begin{align*}
\ell([M, N]^M) \leq \ell(M \t N) \leq \ell([M, N]^M) + 1, \\
\ell([M, N]^N) \leq \ell([M, N]^M) + 1.
\end{align*}

\item[(iii)] If $[M, N]^M$ is Engel, then $M \t N$ and $[M, N]^N$ are Engel too. Moreover, if $[M, N]^M$ is $n$-Engel, then $M \t N$ and $[M, N]^N$ are $(n+1)$-Engel.
\end{enumerate}
\end{theorem}

\section{Universal central extensions of Lie superalgebras}

Now we use the non-abelian tensor product of Lie superalgebras to describe universal central extensions of Lie superalgebras.
Recall that a central extension $\mathfrak{u} \colon U \twoheadrightarrow P$ is universal if for any other central extension
$f \colon M \twoheadrightarrow P$ there is a unique homomorphism $\theta \colon U \to M$ such that
$f \circ \theta = \mathfrak{u}$. It is shown in \cite{Neh} that a Lie superalgebra $P$ admits a universal central extension
if and only if $P$ is perfect, i.e. $P = [P, P]$.

It follows from Proposition~\ref{P:epim} and Lemma~\ref{L:crosker}(i) that the homomorphism
 $\mathfrak{u} \colon P \t P \twoheadrightarrow [P, P]$, $\mathfrak{u}(p \t p') = [p, p']$, is a central extension of the Lie superalgebra $[P, P]$.

\begin{theorem}\label{T:uce}
If $P$ is a perfect Lie superalgebra, then the central extension $\mathfrak{u} \colon P \t P \twoheadrightarrow P$ is the universal central extension.
\end{theorem}

\begin{proof}
Let $f \colon M \twoheadrightarrow P$ be a central extension of $P$. Since $\Ker f$ is in the centre of $M$,
we get a well-defined homomorphism of Lie superalgebras $\theta \colon P \t P \to M$ given by $\theta(p \t p') = [m_p, m_{p'}]$,
where $m_p$ and $m_{p'}$ are any preimages of $p$ and $p'$, respectively. Obviously $\theta \circ f = \mathfrak{u}$.
Since $P$ is perfect, then by relation (iv) of Definition~\ref{D:tensor}, so is $P \t P$. Then by \cite[Lemma 1.4]{Neh}
the homomorphism $\theta$ is unique.
\end{proof}

\begin{remark}
If $P$ is a perfect Lie superalgebra, then $\Ho_2(P) \approx \Ker (P \t P \overset{\mathfrak{u}}{\longrightarrow} P)$,
 since the kernel of the universal central extension is isomorphic to the second homology $\Ho_2(P)$ (see \cite{Neh}).
\end{remark}

It is a classical result that the universal central extension of the Lie algebra $\mathfrak{sl}(n, A)$, where $A$ is a unital associative algebra,
is the Steinberg algebra $\mathfrak{st}(n, A)$, when $n\geq 5$  (see e. g. \cite{KaLo}).
Recently, in \cite{ChSu, GaLa}, this result has been extended  to Lie superalgebras.
Below, using the non-abelian tensor product of Lie superalgebras, we propose an alternative proof of the same result.

First we recall from \cite{ChSu} that, given a unital associative superalgebra $A$, the \emph{Lie superalgebra} $\mathfrak{sl}(m, n, A)$, $m+n\geq 3$,
is defined to be the subalgebra of the Lie superalgebra $\mathfrak{gl}(m, n, A)$ \big(see Example~\ref{E:sup} (ii)\big)
generated by the elements $E_{ij}(a)$, $1\leq i\neq j\leq m+n$, $a\in A\scero \cup A\suno$.
 It is shown in \cite[Lemma 3.3]{ChSu} that $\mathfrak{sl}(m, n, A)$ is a perfect Lie superalgebra.
  This guarantees the existence of the universal central extension of  $\mathfrak{sl}(m, n, A)$.

\emph{The Steinberg Lie superalgebra} $\mathfrak{st}(m, n, A)$ is defined for $m+n\geq 3$ to be the Lie superalgebra
generated by the homogeneous elements $F_{ij}(a)$, where $1 \leq i \neq j \leq m+n$, $a \in A$ is a homogeneous element
and the $\mathbb{Z}_2$-grading is given by $|F_{ij}(a)|=|i|+|j|+|a|$, subject to the following relations:
\begin{align*}
&a \mapsto F_{ij}(a) \text{ is a }\mathbb{K}\text{-linear map,}\\
&[F_{ij}(a), F_{jk}(b)] = F_{ik}(ab), \text{ for distinct }i, j, k,\\
&[F_{ij}(a), F_{kl}(b)] \, = 0, \text{ for } j \neq k, i \neq l.
\end{align*}

 \begin{theorem}[\cite{ChSu}]
 If $m + n \geq 5$, then  the canonical epimorphism
 \[
 \mathfrak{st}(m, n, A)\twoheadrightarrow \mathfrak{sl}(m, n, A), \quad  F_{ij}(a)\mapsto E_{ij}(a),
 \]
  is the universal central extension of the perfect Lie superalgebra  $\mathfrak{sl}(m, n, A)$.
 \end{theorem}
 \begin{proof}
 We claim that there is an isomorphism of Lie superalgebras
\[
\mathfrak{st}(m, n, A)\cong \mathfrak{st}(m, n, A) \t \mathfrak{st}(m, n, A).
\]
Indeed, one can readily check that the maps
\begin{align*}
 & \mathfrak{st}(m, n, A)\longrightarrow \mathfrak{st}(m, n, A) \t \mathfrak{st}(m, n, A), \quad F_{ij}(a)\mapsto F_{ik}(a)\t F_{kj}(1) \ \text{for} \ k\neq i,j,\\
 &\mathfrak{st}(m, n, A) \t \mathfrak{st}(m, n, A)\longrightarrow \mathfrak{st}(m, n, A), \quad  F_{ij}(a)\t F_{kl}(b)\mapsto [F_{ij}(a),  F_{kl}(b)],
 \end{align*}
are well-defined homomorphisms of Lie superalgebras if $m + n \geq 5$, and they are inverses to each other.
Since $\mathfrak{st}(m, n, A)$ is a perfect Lie superalgebra, then Theorem~\ref{T:uce} and \cite[Corollary 1.9]{Neh} complete the proof.
 \end{proof}

\section{Non-abelian homology of Lie superalgebras}\label{S:nahom}

The low-dimensional non-abelian homology of Lie algebras with coefficients in crossed modules was defined in \cite{Gui}
and it was extended to all dimensions in \cite{InKhLa}. In this section we extend to Lie superalgebras
the construction of zero and first non-abelian homologies. We also relate the non-abelian homology of Lie superalgebras
with the cyclic homology of associative superalgebras studied in \cite{IoKo, Kas}.

\subsection{Construction of the non-abelian homology and some properties}

Let $P$ be a Lie superalgebra. We denote by $\cross(P)$ the category of crossed modules of Lie superalgebras over $P$ (crossed $P$-modules, for short),
whose objects are crossed modules $(M, \partial) \equiv (\partial \colon M \to P)$ and a morphism from $(M, \partial)$ to $(N, \partial')$
is a Lie superalgebra homomorphism $f \colon M \to N$ such that $f(^pm) = \mbox{}^pf(m)$ for all $p \in P$, $m \in M$
and $\partial' \circ f = \partial$. By an  exact sequence
$ (L, \partial'') \xrightarrow{ \; f \; } (M, \partial) \xrightarrow{ \; g \; } (N, \partial') $ in $\cross(P)$
we mean that the sequence of Lie superalgebras  $ L \xrightarrow{ \; f \; }  M \xrightarrow{ \; g \; } N $ is exact.

\begin{lemma}
Given a short exact sequence in $\cross(P)$
\[
0 \to (L, \partial'') \xrightarrow{ \; f \; } (M, \partial) \xrightarrow{ \; g \; } (N, \partial') \to 0,
\]
the morphism $\partial'' \colon L \to P$ is trivial and $L$ is an abelian Lie superalgebra.
\end{lemma}

\begin{proof}
Clearly $\partial'' = \partial' \circ g \circ f = 0$ and $[l, l'] = \acts{\partial''(l)}l' = 0$, for all $l,l'\in L$.
\end{proof}

If $(M, \partial)$ and $(N, \partial')$ are two crossed $P$-modules, then the Lie superalgebras $M$ and $N$
act compatibly on each other via the action of $P$. Thus, we can construct the non-abelian tensor product of Lie superalgebras $M \t N$.
Moreover, we have an action of $P$ on $M \t N$ defined by $\acts{p}(m \t n) = \acts{p}m \t n + \menosuno{p}{m} m \t \acts{p} n$,
and straightforward computations show that $\eta \colon M \t N \to P$, $m \t n \mapsto [\partial(m), \partial'(n)]$, is a crossed $P$-module.

\begin{proposition}\label{P:exact}
Let $(M, \partial)$ be a crossed $P$-module. There is a right exact functor $(M \t -) \colon \cross(P) \to \cross(P)$ given, for any crossed $P$-module $(N, \partial')$, by
\[
(M \t -)(N, \partial') = (M \t N, \eta).
\]
\end{proposition}

\begin{proof}
It is an immediate consequence of Proposition~\ref{P:tenfunc}
\end{proof}

\begin{definition}
Let $(M, \partial)$ be a crossed $P$-module. We define the zero and first non-abelian homologies of $P$ with coefficients in $M$ by setting
\[
\nh_0(P, M) = \Coker \nu \qquad \text{ and } \qquad \nh_1(P, M) = \Ker \nu,
\]
where $\nu \colon P \t M \to M$, $p \t m \mapsto \acts{p}m$, is the Lie superalgebra homomorphism as in Proposition~\ref{P:epim}.
\end{definition}

If we consider the crossed $P$-module $(P, \ide_P)$ we have that
\[
\nh_0(P, P) = \dfrac{P}{[P, P]} \cong \Ho_1(P).
\]
In addition, if $P$ is perfect, by Theorem~\ref{T:uce} we have that $\nh_1(P, P) \cong \Ho_2(P)$.

\

The zero and first non-abelian homologies generalize respectively the zero and first homologies of Lie superalgebras in the sense of the following proposition.
\begin{proposition}
Let the ground ring $\mathbb{K}$ be a field. Let $P$ be a Lie superalgebra and $M$  a $P$-supermodule thought as a crossed $P$-module $(M,0)$.
 Then there are isomorphisms of super vector spaces
\[
\nh_0(P, M)\cong \Ho_0(P,M)\quad \text{and} \quad  \nh_1(P, M)\cong \Ho_1(P,M).
\]
\end{proposition}
\begin{proof}
This is a direct consequence of Proposition~\ref{P:augmentation} and  the isomorphisms \eqref{E:zero_homology} and \eqref{E:first_homology}.
\end{proof}

\begin{proposition}\label{P:long}
Given a short exact sequence in $\cross(P)$
\[
0 \to (L, 0) \to (M, \partial) \to (N, \partial') \to 0
\]
we have an exact sequence of supermodules
\[
\nh_1(P, L) \to \nh_1(P, M) \to \nh_1(P, N) \to \nh_0(P, L) \to \nh_0(P, M) \to \nh_0(P, N) \to 0.
\]
\end{proposition}

\begin{proof}
The proof is an immediate consequence of the snake lemma applied to the diagram obtained from Proposition~\ref{P:exact}
\[
\xymatrix{
         & P \t L \ar[d] \ar[r] & P \t M \ar[d] \ar[r] & P \t N \ar[d] \ar[r] & 0  \\
0 \ar[r] & L             \ar[r] & M             \ar[r] & N             \ar[r] & 0.
}
\]
\end{proof}

\subsection{Application to the cyclic homology of associative superalgebras}
Now we recall from \cite{Kas} and \cite{IoKo} the definition of cyclic homology of associative superalgebras.
 Let $A$ be an associative superalgebra and $(C'_*(A), d'_*)$ denote its Hochschild complex,
 that is $C'_n(A) = A^{\otimes_{\mathbb{K}}(n+1)}$ and the boundary map $d_n \colon C'_n(A) \to C'_{n-1}(A)$ is given by
\begin{multline*}
d'_n(a_0 \t \cdots \t a_n) = \sum_{i=0}^{n-1} (-1)^{i} a_0 \t \cdots \t a_ia_{i+1} \t \cdots \t a_n  \\ + (-1)^{n + |a_n|(|a_0| + \cdots + |a_{n-1}|)}a_na_0 \t \cdots \t a_{n-1}.
\end{multline*}
Now the cyclic group $\mathbb{Z}/(n+1)\mathbb{Z}$ acts on $A^{\t_{\mathbb{K}} (n+1)}$ via
\[
t_n(a_0 \t \cdots \t a_n) = (-1)^{n + |a_n|\sum_{k<n}|a_k|}a_n \t a_0 \t \cdots \t a_{n-1},
\]
where $t_n = 1 + (n+1)\mathbb{Z} \in \mathbb{Z}/(n+1)\mathbb{Z}$. For each $n \geq 0$, consider the quotient
$C_n(A)= A^{\otimes_{\mathbb{K}}(n+1)}/\Im(1 - t_n)$ which is the module of coinvariants of $C_n'(A)$ under
the $\mathbb{Z}/(n+1)\mathbb{Z}$-action. Then $d'_n$ induces a well-defined map $d_n\colon C_n(A)\to C_{n-1}(A)$
and  there is an induced chain complex $(C_*(A), d_*)$, which is called the Connes complex of $A$.
Its homologies are, by definition, the \emph{cyclic homologies of the associative superalgebra} $A$, denoted by $\hc_n(A)$, $n\geq 0$.

Easy calculations show that, given an associative superalgebra $A$,  $\hc_1(A)$ is the kernel of the homomorphism of supermodules
\[
(A \t_{\mathbb{K}} A)/\II(A) \to [A, A], \quad a\t b\mapsto ab -(-1)^{|a||b|}ba,
\]
where $[A, A]$ is the graded submodule of $A$ generated by the elements $ab-(-1)^{|a||b|}ba$ and
$\II(A)$ is the graded submodule of the supermodule $A \t_{\mathbb{K}} A$ generated by the elements
\begin{align*}
&a \t b + \menosuno{a}{b} b \t a,\\
&ab \t c - a \t bc + (-1)^{|c|(|a| + |b|)} ca \t b,
\end{align*}
for all homogeneous $a, b, c \in A$.

Now let us consider $A$ as a Lie superalgebra \big(see Example~\ref{E:sup}(i)\big). Then there is a Lie superalgebra structure on $(A \t_{\mathbb{K}} A)/\II(A)$ given by
\[
[a\t b, a'\t b']= [a, b] \t [a', b']
\]
for all $a,a', b,b'\in A$. We denote this Lie superalgebra by $\VV(A)$. In fact, $\VV(A)$ is the quotient
of the non-abelian tensor product $A \t A$ by the graded ideal generated by the elements $x \t y + \menosuno{x}{y} y \t x$
and $xy \t z - x \t yz + (-1)^{|z|(|x| + |y|)} zx \t y$, for all homogeneous $x, y, z \in A$.

\begin{proposition} Let $A$ be a Lie superalgebra. Then the following assertions hold:
\begin{itemize}[font=\normalfont]
\item[(i)] There are compatible actions of the Lie superalgebras $A$ and $\VV(A)$ on each other.
\item[(ii)] The map $\mu \colon \VV(A) \to A$ given by $x \t y \mapsto [x, y]$, together with the action of $A$ on $\VV(A)$, is a crossed module of Lie superalgebras.
\item[(iii)] The action of $A$ on $\VV(A)$ induces the trivial action of $A$ on $\hc_1(A)$.

\item[(iv)] There is a short exact sequence in the category $\cross(A)$
\[
0 \to (\hc_1(A),0) \to (\VV(A),\mu) \to ([A, A], i) \to 0,
\]
where $i\colon [A,A]\to A$ is the inclusion.

\end{itemize}
\end{proposition}

\begin{proof}\hfill
\begin{enumerate}[leftmargin=0.cm,itemindent=.5cm,labelsep=0.2cm,align=left]
 \item[(i)] The action of $A$ on $\VV(A)$ is induced by the action of $A$ on $A\t A$ given in Proposition~\ref{P:epim}(ii), that is
\begin{align*}
\acts{a}(x \t y) &= [a, x] \t y + \menosuno{a}{x} x \t [a, y] \\
&= ax \t y + (-1)^{|a|(|x| + |y|)} x \t ya - \menosuno{x}{a} x \t ay - \menosuno{x}{a} xa \t y \\
&=a \t xy - \menosuno{x}{y} a \t yx \\
&= a \t [x, y],
\end{align*}
whilst the action of $\VV(A)$ on $A$ is defined by
\[
{}^{x\t y}a=\big[[x,y],a\big]
\]
for all homogeneous $a,x,y\in A$. Straightforward calculations show that these are indeed (compatible) actions of Lie superalgebras.

\item[(ii)] Since the crossed module of Lie superalgebras $A\t A \to A$, $x\t y\mapsto [x,y]$, given in Proposition~\ref{P:epim},
vanishes on the elements of the form $x \t y + \menosuno{x}{y} y \t x$ and $xy \t z - x \t yz + (-1)^{|z|(|x| + |y|)} zx \t y$,
then $\mu$ is well defined and obviously it is a crossed module of Lie superalgebras.

\item[(iii)] If $\sum_i \lambda_i (x_i \t y_i) \in \hc_1(A)$, i.e. $\sum_i \lambda_i [x_i, y_i]=0$, then  for all $a \in A$ we have
\[
\acts{a}\big(\sum_i \lambda_i (x_i \t y_i)\big) = \sum_i \lambda_i(a \t [x_i, y_i]) = a \t \sum_i \lambda_i [x_i, y_i] = 0.
\]

\item[(iv)] This is an immediate consequence of the assertions above. \qedhere
\end{enumerate}
\end{proof}

By Proposition~\ref{P:long} we have the following exact sequence of supermodules
\begin{equation}\label{E:seq2}
\begin{tikzpicture}[baseline=(current  bounding  box.center)]
        \matrix (m) [
            matrix of math nodes,
            row sep=1em,
            column sep=2em,
            text height=2.8ex, text depth=1.5ex
        ]
        { \nh_1\big(A, \hc_1(A)\big)  & \nh_1\big(A, \VV(A)\big) & \nh_1(A, [A, A]) \\
             \nh_0\big(A, \hc_1(A)\big) & \nh_0\big(A, \VV(A)\big) &\nh_0(A, [A, A]) & 0. \\
        };
        \path[overlay,->, font=\scriptsize,>=angle 90]
        (m-1-1) edge (m-1-2)
        (m-1-2) edge (m-1-3)
        (m-1-3) edge[out=355,in=175]  (m-2-1)
        (m-2-1) edge (m-2-2)
        (m-2-2) edge (m-2-3)
        (m-2-3) edge (m-2-4);
\end{tikzpicture}
\end{equation}

Below, we will calculate some of the terms of this exact sequence. At first,
by analogy to the Dennis-Stein generators \cite{DeSt}, we give a definition of the first Milnor cyclic
homology for associative superalgebras.

\begin{definition}
Let $A$ be an associative superalgebra. We define the first Milnor cyclic homology $\hc^M_1(A)$ of $A$ to be the quotient
of the supermodule $A\t_{\mathbb{K}} A$ by the graded ideal generated by the elements
\begin{align*}
&a \t b + \menosuno{a}{b} b \t a,\\
&ab \t c - a \t bc + (-1)^{|c|(|a| + |b|)} ca \t b,\\
&a \t bc -\menosuno{b}{c} a\t cb,
\end{align*}
 for all homogeneous $a,b,c\in A$.
\end{definition}

It is clear that if $A$ is supercommutative, that is, $ab=\menosuno{a}{b}ba$, for all homogeneous $a,b\in A$, then $\hc_1(A) \cong \hc^M_1(A)$.

\begin{lemma}
We have the following equalities and isomorphisms
\begin{itemize}[font=\normalfont]
\item[(i)] $\nh_0\big(A, \hc_1(A)\big) = \hc_1(A)$,
\item[(ii)] $\nh_1\big(A, \hc_1(A)\big) \cong A/[A, A] \t_{\mathbb{K}} \hc_1(A)$,
\item[(iii)] $\nh_0(A, [A, A]) = [A, A]/\big[A, [A, A]\big]$,
\item[(iv)]  $\nh_0\big(A, \VV(A)\big) \cong \hc^M_1(A)$.
\end{itemize}
\end{lemma}

\begin{proof}\hfill
\begin{enumerate}[leftmargin=0.cm,itemindent=.5cm,labelsep=0.2cm,align=left]
\item[(i)] Since $A$ acts trivially on $\hc_1(A)$, we have that $\Coker\big(A \t \hc_1(A) \to \hc_1(A) \big) = \hc_1(A)$.

\item[(ii)] Since $\hc_1(A)$ is abelian, by Proposition~\ref{P:ab} we have that $\Ker\big(A \t \hc_1(A) \to \hc_1(A) \big) \cong A/[A, A] \t_{\mathbb{K}} \hc_1(A)$.

\item[(iii)] and (iv) are straightforward. \qedhere
\end{enumerate}
\end{proof}

It follows that the exact sequence \eqref{E:seq2} can be written as in the following theorem.

\begin{theorem}
If $A$ is a unital associative  superalgebra. Then there is an exact sequence of supermodules

\begin{tikzpicture}
        \matrix (m) [
            matrix of math nodes,
            row sep=1em,
            column sep=2em,
            text height=2.8ex, text depth=1.5ex
        ]
        { \dfrac{A}{[A, A]} \t_{\mathbb{K}} \hc_1(A)  & \nh_1\big(A, \VV(A)\big) & \nh_1(A, [A, A]) \\
             \hc_1(A) & \hc^M_1(A) & \dfrac{[A, A]}{\big[A,[A, A]\big]} & 0. \\
        };
        \path[overlay,->, font=\scriptsize,>=angle 90]
        (m-1-1) edge (m-1-2)
        (m-1-2) edge (m-1-3)
        (m-1-3) edge[out=355,in=175]  (m-2-1)
        (m-2-1) edge (m-2-2)
        (m-2-2) edge (m-2-3)
        (m-2-3) edge (m-2-4);
\end{tikzpicture}
\end{theorem}

\begin{cor}
If $A$ is perfect as a Lie superalgebra, we have an exact sequence
\[
0 \to \nh_1\big(A, \VV(A)\big) \to \Ho_2(A) \to \hc_1(A) \to 0,
\]
where $\Ho_2(A)$ is the usual second homology of the Lie superalgebra $A$. If in addition $\Ho_2(A) = 0$, then all terms of the exact sequence in the previous theorem are trivial.
\end{cor}

\begin{proof}
Since $A$ is perfect we know that $\nh_1(A, A) \cong \Ho_2(A)$, $A/[A, A] \t_{\mathbb{K}} \hc_1(A) = 0$ and the map $A \t \VV(A) \to \VV(A)$ is surjective.
\end{proof}

\section{Non-abelian exterior product of Lie superalgebras}
In this section we extend to Lie superalgebras the definition of the non-abelian exterior product of Lie algebras introduced in \cite{Ell1}.
Then we use it to derive the Hopf formula for the second homology of a Lie superalgebra and to construct a six-term exact homology sequence of Lie superalgebras.
\subsection{Construction of the non-abelian exterior product}

Let $P$ be a Lie superalgebra and $(M,\partial)$ and $(N,\partial')$  two crossed $P$-modules. We consider the actions of $M$ and $N$ on each other via $P$.

\begin{lemma}
Let $M \square N$ be the graded submodule of $M \t N$ generated by the elements
\begin{enumerate}[font=\normalfont]
\item[(a)] $m \t n + \menosuno{m'}{n'} m' \t n'$, where $\partial(m) = \partial'(n')$ and $\partial(m') = \partial'(n)$,
\item[(b)] $m_{\bar{0}} \t n_{\bar{0}}$, where $\partial(m_{\bar{0}}) = \partial'(n_{\bar{0}})$,
\end{enumerate}
with $m, m' \in M\scero \cup M\suno$, $n, n' \in N\scero \cup N\suno$, $m\scero \in M\scero$ and $n\scero \in N\scero$.
Then, $M \square N$ is a graded ideal in the centre of $M \t N$.
\end{lemma}

\begin{proof}
Given an element  $m \t n + \menosuno{m'}{n'} m' \t n'$ of the form (a), suppose that $|m'| = |n|$, then we have
\begin{align*}
[x \t y, m \t n + \menosuno{m'}{n'} m' \t n'] &= - \menosuno{x}{y} (\acts{y}x) \t \big(\acts{m}n + \menosuno{m'}{n'}(\acts{m'}\!n')\big) \\
{} &= - \menosuno{x}{y} (\acts{y}x) \t \big(\acts{\partial(m)}n + \menosuno{m'}{n'} (\acts{\partial(m')}n')  \big) \\
{} &= - \menosuno{x}{y} (\acts{y}x) \t \big(\acts{\partial'(n')}n + \menosuno{m'}{n'} (\acts{\partial'(n)}n')  \big) \\
{} &= - \menosuno{x}{y} (\acts{y}x) \t \big([n', n] + \menosuno{n}{n'} [n, n']  \big) \\
{} &= 0.
\end{align*}
This is also true when $|m'| \neq |n|$. Indeed, if $|m'| \neq |n|$, since $\partial$, $\partial'$ are even maps,
the equality $\partial(m) = \partial'(n')$ holds if and only if $\partial(m) = 0 = \partial'(n')$.
Now take an element $m\scero \t n\scero$ of the form (b). Then we have
\begin{align*}
[x \t y, m\scero \t n\scero] &=  - \menosuno{x}{y} (\acts{y}x) \t (\acts{m\scero}n\scero) \\
{} &= - \menosuno{x}{y} (\acts{y}x) \t (\acts{\partial(m\scero)}n\scero) \\
{} &= - \menosuno{x}{y} (\acts{y}x) \t (\acts{\partial(n\scero)}n\scero) \\
{} &= - \menosuno{x}{y} (\acts{y}x) \t [n\scero, n\scero] \\
{} &= 0,
\end{align*}
for any $x \t y \in M \t N$. This completes the proof.
\end{proof}

\begin{definition} Let $P$ be a Lie superalgebra and $(M,\partial)$ and $(N,\partial')$  two crossed $P$-modules.
The \emph{non-abelian exterior product $M \wedge N$} of the Lie superalgebras $M$ and $N$ is defined by
\[
M \wedge N = \dfrac{M \t N}{M \square N}.
\]
The equivalence class of $m \t n$ will be denoted by $m \wedge n$.
\end{definition}

Note that if $M = M\scero$ and $N = N\scero$ then $M \wedge N$ coincides with the non-abelian exterior product of Lie algebras \cite{Ell1}.

Reviewing Section~\ref{S:tensor}, one can easily check that most of results on the non-abelian tensor product are fulfilled for the non-abelian exterior product.
In particular, there are homomorphisms of Lie superalgebras $M \wedge N \to M$, $M \wedge N \to N$ and actions of
 $M$ and $N$ on $M \wedge N$, induced respectively by the homomorphisms and actions given in Proposition~\ref{P:epim}. It is also satisfied the isomorphism $M \wedge N \cong  N \wedge M$.
  Further, given a short exact sequence of Lie superalgebras
$ 0 \to K \to M \to P \to 0$,
as an exterior analogue of the exact sequence (\ref{C:exact}), we get the following exact sequence of Lie superalgebras
\begin{equation}\label{eq_exact_exterior}
K \wedge M \to M \wedge M \to P \wedge P \to 0.
\end{equation}

Given a Lie superalgebra $M$, since $\ide \colon M \to M$ is a crossed module, we can consider $M \wedge M$. It is the quotient of $M \t M$ by the following  relations
\begin{align*}
m \wedge m' & = -\menosuno{m}{m'}m' \wedge m,\\
m\scero \wedge m\scero & = 0,
\end{align*}
for all $m, m' \in M\scero \cup M\suno$ and $m\scero \in M\scero$.
In the particular case when $M$ is perfect, it is easy to see that $M \square M = 0$, so $M \wedge M \cong M \t M$ and in Theorem~\ref{T:uce} we can replace $M \t M$ by $M \wedge M$.

\subsection{A six term exact homology sequence}
In \cite{Ell2}, the non-abelian exterior product of Lie algebras is used to construct a six-term exact sequence of homology of Lie algebras.
In this section we will extend these results to Lie superalgebras.

First of all, we prove an analogue of Miller's theorem \cite{Mil} on free Lie superalgebras extending the similar result obtained in \cite{Ell2} for Lie algebras.

\begin{proposition}\label{P:free}
Let $F = \efe(X)$ be the free Lie superalgebra on a graded set $X$. Then the homomorphism $F \wedge F \to F$, $x \wedge y \mapsto xy$ is injective.
\end{proposition}

\begin{proof}
Let us prove that $[F, F] \cong F \wedge F$. Using the same notations as in Construction~\ref{C:free},
we define a map $\phi \colon \alg(X) * \alg(X) \to F \wedge F$ by $\sum_i\lambda_ix_iy_i \mapsto \sum_i \lambda_i (x_i \wedge y_i)$,
where $\alg(X) * \alg(X)$ is the free product of superalgebras.
It is easy to see that $\phi$ is a $\mathbb{K}$-superalgebra homomorphism since $[x \wedge y, x' \wedge y'] = xy \wedge x'y'$.
The ideal $I$ is contained in $\alg(X) * \alg(X)$ and by using the defining relations of $F \wedge F$
it is not difficult to check that $\phi$ vanishes on $I$. So we have an induced map from $[F, F]$ to $F \wedge F$,
which is inverse to the homomorphism $F \wedge F \to [F, F]$, $x \wedge y \mapsto xy$.
\end{proof}

Let $P$ be a Lie superalgebra and take the quotient supermodule $(P\wedge_{\mathbb{K}}P)/\Im d_3$,
where $d_3\colon \bigwedge_{\mathbb{K}}^3(P) \to \bigwedge_{\mathbb{K}}^2(P)$ is the boundary map
in the homology complex $(C_*(P,\mathbb{K}), d_*)$. Here $\mathbb{K}$ is considered as a trivial $P$-module.
We define a bracket in $(P\wedge_{\mathbb{K}}P)/\Im d_3$ by setting
\[
[x\wedge y, x'\wedge y']=[x,y]\wedge [x',y']
\]
for all $x,y\in P$. As a particular case of the exterior analogue of Proposition~\ref{P:tensoriso} we have
\begin{lemma}\label{Lemma d3}
There is an isomorphism of Lie superalgebras
\[
\dfrac{P\wedge_{\mathbb{K}}P}{\Im d_3}\approx P\wedge P.
\]
\end{lemma}

\begin{cor}\label{C:hopf}\hfill
\begin{itemize}[font=\normalfont]
\item[(i)] For any Lie superalgebra $P$ there is an isomorphism of supermodules
\[
H_2(P)\cong \Ker(P\wedge P\to P).
\]
\item[(ii)] $H_2(F)=0$ if $F$ is a free Lie superalgebra.
\item[(iii)] (Hopf Formula) Given a free presentation
$ 0\to R\to F\to P\to 0 $
of a Lie superalgebra $P$, there is an isomorphism of supermodules
\[
H_2(P)\cong \dfrac{R\cap [F,F]}{[F,R]}.
\]
\end{itemize}
\end{cor}
\begin{proof}\hfill
\begin{enumerate}[leftmargin=0.cm,itemindent=.5cm,labelsep=0.2cm,align=left]
\item[(i)] This follows immediately from Lemma~\ref{Lemma d3}.

\item[(ii)] This is a consequence of (i) and Proposition~\ref{P:free}.

\item[(iii)] Since $F\wedge F\cong [F,F]$, using the exact sequence \eqref{eq_exact_exterior}, we have
\[
P\wedge P\cong \dfrac{[F,F]}{[F,R]}.
\]
Then Lemma~\ref{Lemma d3} completes the proof. \qedhere
\end{enumerate}
\end{proof}

\begin{theorem}
Let $M$ be a graded ideal of a Lie superalgebra $P$. Then there is an exact sequence
\[
\Ker(P \wedge M \to P) \to \Ho_2(P) \to \Ho_2(P/M) \to \dfrac{M}{[P, M]}
\to \Ho_1(P) \to \Ho_1(P/M) \to 0.
\]
\end{theorem}

\begin{proof}
By using the exact sequence \eqref{eq_exact_exterior} we have the following commutative diagram of Lie superalgebras with exact rows
\[
\xymatrix{
         & M \wedge P \ar[d] \ar[r] & P \wedge P \ar[d] \ar[r] & \dfrac{P}{M}  \wedge \dfrac{P}{M} \ar[d] \ar[r] & 0  \\
0 \ar[r] & M             \ar[r] & P             \ar[r] & \dfrac{P}{M}             \ar[r] & 0.
}
\]
Since $\Coker(M \wedge P \cong P \wedge M \to M) \cong M/[P, M]$ and $\Coker(P \wedge P \to P) \cong P / [P, P] \cong \Ho_1(P)$,
then the assertion follows by using snake lemma and Corollary~\ref{C:hopf}(i).
\end{proof}

In particular, if $P$ is a Lie algebra and $M$ is an ideal of $P$, then this sequence coincides with the six-term exact sequence in the homology of Lie algebras obtained in \cite{Ell2}.

\section*{Acknowledgements}
 The authors wish to thank the anonymous referee for his help in improving
the presentation of this paper.
The authors were supported by Ministerio de Economía y Competitividad (Spain), grant MTM2013-43687-P (European FEDER support included).
The first and third authors were also supported by Xunta de Galicia, grant GRC2013-045 (European FEDER support included).
The first author was also supported  by FPU scholarship, Ministerio de Educación, Cultura y Deporte  (Spain).
The second author was also supported by Xunta de Galicia, grant EM2013/016 (European FEDER support included) and
Shota Rustaveli National Science Foundation, grant DI/12/5-103/11.


\begin{thebibliography}{10}


\bibitem{BrLo}
R.~Brown, J.-L. Loday, Van {K}ampen theorems for diagrams of spaces, Topology
  26~(3) (1987) 311--335, with an appendix by M. Zisman.

\bibitem{CCF}
C.~Carmeli, L.~Caston, R.~Fioresi, Mathematical foundations of supersymmetry,
  EMS Series of Lectures in Mathematics, European Mathematical Society (EMS),
  Z\"urich, 2011.

\bibitem{CaLa}
J.~M. Casas, M.~Ladra, Perfect crossed modules in {L}ie algebras, Comm. Algebra
  23~(5) (1995) 1625--1644.

\bibitem{CGL}
J.~Castiglioni, X.~Garc\'ia-Mart\'inez, M.~Ladra, Universal central extensions
  of {L}ie--{R}inehart algebras, arXiv:1403.7159 (2014).

\bibitem{ChSu}
H.~Chen, J.~Sun, Universal central extensions of $sl_{m|n}$ over
  {$\mathbb{Z}/2\mathbb{Z}$}-graded algebras, J. Pure Appl. Algebra 219
  (2015) 4278--4294.

\bibitem{DeSt}
R.~K. Dennis, M.~R. Stein, {$K_{2}$} of discrete valuation rings, Advances in
  Math. 18~(2) (1975) 182--238.

\bibitem{Ell2}
G.~J. Ellis, Nonabelian exterior products of {L}ie algebras and an exact
  sequence in the homology of {L}ie algebras, J. Pure Appl. Algebra 46~(2-3)
  (1987) 111--115.


\bibitem{Ell1}
G.~J. Ellis, A nonabelian tensor product of {L}ie algebras, Glasgow Math. J.
  33~(1) (1991) 101--120.

\bibitem{GaLa}
X.~Garc\'ia-Mart\'inez, M.~Ladra, Universal central extensions of
  $\mathfrak{sl}(m, n, a)$ of small rank over associative superalgebras,
  arXiv:1405.4035 (2014).

\bibitem{Gui}
D.~Guin, Cohomologie des alg\`ebres de {L}ie crois\'ees et {$K$}-th\'eorie de
  {M}ilnor additive, Ann. Inst. Fourier (Grenoble) 45~(1) (1995) 93--118.

\bibitem{HiSt}
P.~J. Hilton, U.~Stammbach, A course in homological algebra, vol.~4 of Graduate
  Texts in Mathematics, 2nd ed., Springer-Verlag, New York, 1997.

\bibitem{InKhLa}
N.~Inassaridze, E.~Khmaladze, M.~Ladra, Non-abelian homology of {L}ie algebras,
  Glasg. Math. J. 46~(2) (2004) 417--429.


\bibitem{IoKo}
K.~Iohara, Y.~Koga, Second homology of {L}ie superalgebras, Math. Nachr.
  278~(9) (2005) 1041--1053.


\bibitem{Kac}
V.~G. Kac, Lie superalgebras, Advances in Math. 26~(1) (1977) 8--96.

\bibitem{Kas}
C.~Kassel, A {K}\"unneth formula for the cyclic cohomology of {${\bf
  Z}/2$}-graded algebras, Math. Ann. 275~(4) (1986) 683--699.


\bibitem{KaLo}
C.~Kassel, J.-L. Loday, Extensions centrales d'alg\`ebres de {L}ie, Ann. Inst.
  Fourier (Grenoble) 32~(4) (1982) 119--142.

\bibitem{Mil}
C.~Miller, The second homology group of a group; relations among commutators,
  Proc. Amer. Math. Soc. 3 (1952) 588--595.

\bibitem{Mus}
I.~M. Musson, Lie superalgebras and enveloping algebras, vol. 131 of Graduate
  Studies in Mathematics, American Mathematical Society, Providence, RI, 2012.

\bibitem{Neh}
E.~Neher, An introduction to universal central extensions of {L}ie
  superalgebras, in: Groups, rings, {L}ie and {H}opf algebras ({S}t. {J}ohn's,
  {NF}, 2001), vol. 555 of Math. Appl., Kluwer Acad. Publ., Dordrecht, 2003,
  pp. 141--166.

\bibitem{STME}
A.~R. Salemkar, H.~Tavallaee, H.~Mohammadzadeh, B.~Edalatzadeh, On the
  non-abelian tensor product of {L}ie algebras, Linear Multilinear Algebra
  58~(3-4) (2010) 333--341.


\bibitem{Sch}
M.~Scheunert, The theory of {L}ie superalgebras. An introduction, vol. 716 of
  Lecture Notes in Mathematics, Springer, Berlin, 1979.

\bibitem{Tan}
J.~Tanaka, On homology and cohomology of {L}ie superalgebras with coefficients
  in their finite-dimensional representations, Proc. Japan Acad. Ser. A Math.
  Sci. 71~(3) (1995) 51--53.

 \bibitem{ZL}
 T.~Zhang,  Z.~Liu,
Omni-Lie superalgebras and Lie 2-superalgebras,
Front. Math. China 9~(5) (2014) 1195--1210.

\end{thebibliography}

\end{document}